\documentclass[a4paper,oneside,10pt, reqno]{amsproc}
\usepackage{amsthm}
\usepackage{amsfonts}
\usepackage{amsmath}
\usepackage{amssymb}
\usepackage{mathrsfs}
\usepackage{epsfig}
\usepackage{graphicx}
\usepackage{epstopdf}
\usepackage{color}
\usepackage{ifthen}
\usepackage{float}

\makeatletter
\@namedef{subjclassname@2010}{%
\textup{2010} Mathematics Subject Classification}
\makeatother

\newtheorem{theorem}{Theorem}[section]
\newtheorem{proposition}[theorem]{Proposition}
\newtheorem{corollary}[theorem]{Corollary}
\newtheorem{lemma}[theorem]{Lemma}

\theoremstyle{remark}
\newtheorem{remark}[theorem]{Remark}

\theoremstyle{definition}
\newtheorem{definition}[theorem]{Definition}

\newtheorem{List}[theorem]{List}

\def \dep{\mathsf{d}}

\newcommand{\bq}{\begin{equation}}
\newcommand{\eq}{\end{equation}}
\newcommand{\beqn}{\begin{eqnarray*}}
\newcommand{\eeqn}{\end{eqnarray*}}
\newcommand{\beq}{\begin{eqnarray}}
\newcommand{\eeq}{\end{eqnarray}}

\newcommand{\rar}{\rightarrow}

\newcommand{\bc}{\begin{centre}}
\newcommand{\ec}{\end{centre}}

\newcommand{\ba}{\begin{array}}
\newcommand{\ea}{\end{array}}

\newcommand{\inp}[2]{\langle{#1},\,{#2} \rangle}

\renewcommand{\Delta}{{\nabla}}

\newcommand{\mf}{\mathfrak}

\newcommand*{\child}[1]{\mathsf{Chi}(#1)}
\newcommand*{\childn}[2]{{\mathsf{Chi}}^{\langle#1\rangle}(#2)}

\newcommand*{\Ge}{\geqslant}
\newcommand*{\lambdab}{\boldsymbol\lambda}
\newcommand*{\Le}{\leqslant}

\begin{document}
\title[Commutants and Reflexivity]
{Commutants and Reflexivity  of Multiplication tuples on Vector-valued Reproducing Kernel Hilbert Spaces}
\author{Sameer Chavan \and Shubhankar Podder \and Shailesh Trivedi
}\address{Department of Mathematics and Statistics\\
Indian Institute of Technology Kanpur, India}
   \email{chavan@iitk.ac.in}
   \email{shailtr@iitk.ac.in}
  \address{School of Mathematics, Harish-Chandra Research Institute,
Chhatnag Road \\ Jhunsi, Allahabad 211019, India}
\email{shubhankarpodder@hri.res.in}


\thanks{The research of the third author was supported by the National Post-doctoral Fellowship (Ref. No. PDF/2016/001681), SERB}
   

\date{}

\begin{abstract}
Motivated by the theory of weighted shifts on directed trees and its multivariable counterpart, we address the question of identifying commutant and reflexivity of the multiplication $d$-tuple $\mathscr M_z$ on a reproducing kernel Hilbert space $\mathscr H$ of $E$-valued holomorphic functions on $\Omega$, where $E$ is a separable Hilbert space and $\Omega$ is a bounded domain in $\mathbb C^d$ admitting bounded approximation by polynomials. 
In case $E$ is a finite dimensional cyclic subspace for $\mathscr M_z$, under some natural conditions on the $B(E)$-valued kernel associated with $\mathscr H$,  the commutant of $\mathscr M_z$ is shown to be the algebra $H^{\infty}_{_{B(E)}}(\Omega)$ of bounded holomorphic $B(E)$-valued functions on $\Omega$, provided $\mathscr M_z$ satisfies the matrix-valued von Neumann's inequality. 
This generalizes a classical result of Shields and Wallen (the case of $\dim E=1$ and $d=1$).  
 As an application, we determine the commutant of a Bergman shift on a leafless, locally finite, rooted directed tree $\mathscr T$ of finite branching index. As the second main result of this paper, we show that a multiplication $d$-tuple $\mathscr M_z$ on $\mathscr H$ satisfying the von Neumann's inequality is reflexive. 
This provides several new classes of examples as well as recovers special cases of various known results in one and several variables. 
We also exhibit a family of tri-diagonal $B(\mathbb C^2)$-valued kernels for which the associated multiplication operators $\mathscr M_z$ are non-hyponormal reflexive operators with commutants  equal to $H^{\infty}_{_{B(\mathbb C^2)}}(\mathbb D)$.
\end{abstract}

\maketitle

\section{Introduction}

This paper is motivated by some recent developments pertaining to the function theory of weighted shifts on rooted directed trees and its multivariable counterpart (refer to \cite{Jablonski, CT, MS, B-D-P, CPT, CPT-1}).
In particular, it is
centered on the investigation of two topics from classical function-theoretic operator theory, namely commutants and reflexivity of multiplication tuples on  reproducing kernel Hilbert spaces of vector-valued holomorphic functions (refer to \cite{S, RR, Co1} for a comprehensive account on commutants and reflexivity of unilateral weighted shifts and multiplication operators on reproducing kernel Hilbert spaces of scalar-valued holomorphic functions; refer to \cite{Pt, Co1, Di} for a masterful exposition on reflexivity of algebras of commuting operators). Via a construction of Shimorin \cite{Shimorin}, any bounded linear left-invertible weighted shift on a rooted directed tree can be modeled as the operator of multiplication by the coordinate function on a reproducing kernel Hilbert space of $E$-valued holomorphic functions \cite{CT}, where $E$ is a separable Hilbert space. On the other hand, a classical result of Shields and Wallen \cite[Theorem 2]{Sh-Wa} identifies the commutant of a contractive multiplication operator on a reproducing kernel Hilbert space of scalar-valued holomorphic functions on the open unit disc with the algebra of bounded holomorphic functions (see \cite[Theorem 1]{Sa}, \cite[Chapter II, Theorem 5.4]{Co}, \cite[Chapter VI, Corollary 3.7]{SF} for its variants). This provides essential motivation for vector-valued analog of the aforementioned theorem of Shields and Wallen (see \cite[Theorem 3]{Sa} for a vector-valued version of \cite[Theorem 1]{Sa}). Essential ingredients in the proof of the main result of Section 3 (Theorem \ref{commutant}) include an adaptation of the technique from \cite{Sh-Wa} to the present situation, matrix-valued version of von Neumann's inequality \cite{P}, and the role of the simultaneous boundedness of reproducing kernel and its inverse along the diagonal (cf. \cite[Theorem 5.2]{C-S}). Theorem \ref{commutant} is applicable to the so-called Bergman shifts on locally finite, rooted directed trees of finite branching index (see Proposition \ref{comm-B}). It is worth noting that the simultaneous growth of associated Bergman kernel and its inverse along the diagonal is at most of polynomial order (the reader is referred to \cite{E} and \cite{K}, where asymptotic behavior of scalar-valued kernels has been studied). 

The second main result of this paper (Theorem \ref{reflexive}) ensures reflexivity of multiplication tuple $\mathscr M_z$ on a reproducing kernel Hilbert space $\mathscr H$ of vector-valued  holomorphic functions with essentially the only assumption that $\mathscr M_z$ satisfies the von Neumann's inequality. This provides several new classes of examples and recovers special cases of various known results in one and several variables, see \cite[Theorem 3]{Sa0}, \cite[Section 10, Proposition 37]{S},  \cite[Theorem 15]{B0}, \cite{BC}, \cite[Theorem 5.2]{KP},  \cite[Theorem 4.2]{Se}, \cite[Section 0, Theorem 4]{Mc}, \cite[Theorem A]{AMu} (cf. \cite[Proposition 4.4]{CEP}, \cite[Theorem 3]{OT}, \cite[Theorem 3]{Mc}, \cite[Theorem 3.1]{KOP},  \cite[Corollary 3.7]{Es1}, \cite[Theorem 1.2]{AP}, \cite[Theorem 2.11]{Es2}, \cite[Corollary 7]{Es}).  
It is worth noting that the techniques employed in the proofs of Theorems \ref{commutant} and \ref{reflexive} have some common features (e.g. bounded approximation by polynomials in the sense of \cite{Sh-Wa} and \cite{Mc}).
We conclude the paper by exhibiting a two-parameter family of tri-diagonal matrix-valued kernels to which Theorems \ref{commutant} and \ref{reflexive} are applicable.

We set below the notations used throughout this text.
For a set $X$ and integer $d$, $\mbox{card}(X)$ denotes the cardinality of $X$ and $X^d$ stands for the $d$-fold Cartesian product of $X$.  
The symbol ${\mathbb N}$ stands for the set of nonnegative
integers, while
$\mathbb C$ denotes the field of complex numbers. 
For $\alpha =
(\alpha_1, \ldots, \alpha_d) \in {\mathbb{N}}^d,$ 
we use $\alpha!:=\prod_{j=1}^d \alpha_j!$ and $|\alpha|:=\sum_{j=1}^d
\alpha_j$.
For $w=(w_1, \ldots, w_d) \in \mathbb C^d$ and $\alpha=(\alpha_1, \ldots, \alpha_d) \in \mathbb N^d$, the complex conjugate $\overline{w} \in \mathbb C^d$ of $w$ is given by $(\overline{w}_1, \ldots, \overline{w}_d),$ while $w^\alpha$ denotes the complex number $\prod_{j=1}^d w^{\alpha_j}_j$. The symbol $\mathbb D^d$ is reserved for the open unit polydisc in $\mathbb C^d$ centered at the origin, while the open unit ball in $\mathbb C^d$ centered at the origin is denoted by $\mathbb B^d.$
In case $d=1,$ we prefer the notation $\mathbb D$ in place of $\mathbb D^1$ or $\mathbb B^1.$
Let $\mathcal H$ be a complex Hilbert space. 
If $F$ is a subset of $\mathcal H$, the closure of $F$ is denoted by $\overline{F}$, while the closed linear span of $F$ is denoted by $\bigvee \{x : x \in F\}$. In case $F$ is single-ton $\{x\}$, then $\bigvee \{x\}$  is denoted by the simpler notation $[x]$. If $\mathcal M$ is a finite dimensional subspace of $\mathcal H$, then $\dim \mathcal M$ denotes the vector space dimension of $\mathcal M.$ For a closed subspace $\mathcal M$ of $\mathcal H$,  the orthogonal projection of $\mathcal H$ onto $\mathcal M$ is denoted by $P_{\mathcal M}$.
For a positive integer $d,$ the orthogonal direct
sum of $d$ copies of $\mathcal H$ is denoted by $\mathcal H^{(d)}.$ 
Let ${B}({\mathcal H})$ denote the unital Banach algebra of
bounded linear operators on $\mathcal H.$ The multiplicative identity $I$ of $B(\mathcal H)$ is sometimes denoted by $I_{\mathcal H}$.
For a subspace $\mathcal M$ of $B(\mathcal H)$,  $\overline{\mathcal M}^{\mbox{\tiny WOT}}$ denotes the closure of $\mathcal M$ in the weak operator topology in $B(\mathcal H)$.
For clarity, the norm $\|\cdot\|$ on a normed linear space $X$ is occasionally denoted by $\|\cdot\|_{X}.$ Sometimes, this is denoted by the pair $(X, \|\cdot\|_X)$.
For  a subset $\Omega$ of $\mathbb C^d$ and a normed linear space $X,$ the sup norm of  a function $\Phi : \Omega \rar X$ is given by $\|\Phi\|_{\infty, \Omega}:=\sup_{w \in \Omega}\|\Phi(w)\|_{_{X}}.$
If $T \in B(\mathcal H)$, then $\ker(T)$ denotes the kernel of $T$, $T(\mathcal H)$ denotes the range of $T,$ $T^*$ denotes the Hilbert space adjoint of $T$, while $T^{(d)} \in B(\mathcal H^{(d)})$ denotes for the orthogonal direct sum of $d$ copies of $T$. Given $x, y \in \mathcal H,$ by the rank one operator $x \otimes y$, we understand the bounded linear operator $x \otimes y(h)=\inp{h}{y}x$, $h \in \mathcal H.$

An operator $T \in B(\mathcal H)$ is {\it left-invertible} if $T^*T$ is invertible in $B(\mathcal H)$. The {\it Cauchy dual} of a left-invertible operator $T \in B(\mathcal H)$ is given by $T':=T(T^*T)^{-1}.$ We say that $T \in B(\mathcal H)$ is
{\it analytic} if $\cap_{n \in \mathbb N}T^n(\mathcal H)=\{0\}$. 
An operator $T \in B(\mathcal H)$ is {\it irreducible} if it does not admit a proper reducing subspace.
By a  {\it commuting $d$-tuple $T=(T_1, \ldots, T_d)$ in $B(\mathcal H)$}, we mean a collection of commuting operators $T_1, \ldots, T_d$ in $B(\mathcal H).$
The notations $\sigma(T),$ $\sigma_H(T)$ and $\sigma_p(T)$ are reserved for the Taylor spectrum, Harte spectrum and joint point spectrum of a commuting $d$-tuple $T$ respectively.
The Hilbert adjoint of the commuting $d$-tuple $T=(T_1, \ldots, T_d)$ is the $d$-tuple $T^*=(T^*_1, \ldots, T^*_d),$ and
the joint kernel $\cap_{j=1}^d \ker(T_j)$ of $T$ is denoted by $\ker(T).$
 Further, for $\lambda=(\lambda_1, \ldots, \lambda_d) \in \mathbb C^d$, by $T-\lambda$, we understand the commuting $d$-tuple $(T_1-\lambda_1I_{\mathcal H}, \ldots, T_d-\lambda_dI_{\mathcal H})$. A commuting $d$-tuple $T=(T_1, \ldots, T_d)$ is said to be a {\it contraction} (resp. a {\it joint contraction}) if $T^*_j T_j \Le I$ for every $j=1, \ldots, d$ (resp. $\sum_{j=1}^d T^*_j T_j \Le I$).
The {\it commutant} $\mathscr S'$ of a subset $\mathscr S$ of $B(\mathcal H)$ is given by
\beqn
\mathscr S' := \{T \in B(\mathcal H) : ST=TS~\mbox{for all~}S \in \mathscr S\}.
\eeqn 
For a commuting $d$-tuple $T=(T_1, \ldots, T_d)$ in $B(\mathcal H)$, we use the simpler notation $\{T\}'$ for $\mathscr S',$ where $\mathscr S=\{T_1, \ldots, T_d\}.$ 
Note that $\mathscr S'$ is a unital closed subalgebra of $B(\mathcal H).$
If $C \in B(\mathcal H)$, then Lat\,$C$ denotes the set of all  closed linear subspaces of $\mathcal H$ that are invariant under $C.$ 
Let $\mathscr W$ be a subalgebra of ${B}({\mathcal H})$ containing the identity operator $I_{{\mathcal H}}$, and let Lat\,$\mathscr W$ be the set of all  closed linear subspaces of $\mathcal H$ that are invariant under every operator  $W\in \mathscr W$. The set
\begin{center}
AlgLat\,$\mathscr W=\{C\in {B}(\mathcal H): {\rm Lat}\,\mathscr W \subseteq {\rm Lat}\,C\}$
\end{center}
turns out to be a WOT-closed subalgebra of ${B}(\mathcal H),$ which contains $\mathscr W$. We say that $\mathscr W$ is \emph{reflexive} if
$\mathscr W = \mbox{AlgLat}\,\mathscr W.$
For a commuting $d$-tuple $T=(T_1, \ldots, T_d)$ in ${B}({\mathcal H})$, let $\mathscr W_{T}$ stand for the WOT-closed  subalgebra of $B(\mathcal H) $ generated by $T_1, \ldots, T_d$ and the identity operator $I_{{\mathcal H}}$: 
\beqn
\mathscr W_{T} =\overline{\{p(T) : p \in \mathbb C[z_1, \ldots, z_d]\}}^{\mbox{\tiny WOT}},
\eeqn
where $\mathbb C[z_1, \ldots, z_d]$ denotes the vector space of complex polynomials in $z_1, \ldots, z_d$ and $p(T)$ is given by the polynomial functional calculus of $T$.
A commuting $d$-tuple $T$ is {\it reflexive} if $\mathscr W_T$ is reflexive.

Here is the outline of the paper. In Section 2, we collect essential facts pertaining to the operator-valued kernels and associated reproducing kernel Hilbert spaces. Further, we formally introduce the notion of functional Hilbert space and discuss some properties of associated multiplication tuple.
Sections 3 and 4 are devoted to main results of this paper (and their immediate consequences) on commutants and reflexivity of multiplications tuples on functional Hilbert spaces respectively. In the final section, we discuss applications of the main results to the theory of weighted shifts on rooted directed trees.  Among various applications, we derive the curious fact that  the commutant of a Bergman shift on a leafless, locally finite rooted directed tree $\mathscr T$ of finite branching index is abelian if and only if $\mathscr T$ is graph isomorphic to the rooted directed tree without any branching vertex.

\section{Operator-valued Reproducing Kernels}

Before we introduce the so-called functional Hilbert spaces, we briefly recall from \cite{AM}, \cite{PR} some definitions 
and facts pertaining to Hilbert spaces associated with operator-valued kernels. 
Let $E$ be a Hilbert space and let $X$ be any set. A {\it weak $B(E)$-valued kernel} on $X$ is a function $\kappa : X \times X \rar B(E)$ such that, for any finite set $\{\lambda_1, \ldots, \lambda_n\} \subseteq X$ and any vectors $v_1, \ldots, v_n \in E,$ we have
\beqn
\sum_{i, j=1}^n \inp{\kappa(\lambda_i, \lambda_j)v_j}{v_i}_E \Ge 0.
\eeqn 
If, in addition, $\kappa(\lambda, \lambda) \neq 0$ for any $\lambda \in X$, then $\kappa$ is referred to as a {\it $B(E)$-valued kernel on $X$}. 


With any $B(E)$-valued kernel $\kappa : X \times X \rar B(E)$, one can associate a Hilbert space $\mathscr H$ of $E$-valued functions on $X$ such that for every $\lambda \in X,$
\begin{enumerate}
\item[$(\mathsf C1)$] ~the evaluation at $\lambda$ is a continuous linear function from $\mathscr H$ to $E$,
\item[$(\mathsf C2)$] ~$\{f(\lambda) : f \in \mathscr H\} \neq \{0\}$
\end{enumerate}
(see \cite[Theorem 2.60]{AM}, \cite[Theorem 6.12]{PR}).
In this case, for any $g \in E$ and $\lambda \in X,$ 
\begin{align} \label{rp}
 \left.
  \begin{minipage}{45ex}
\beqn 
& & \kappa(\cdot,\lambda)g  \in  \mathscr H, \\
& & \langle f,\kappa(\cdot,\lambda)g \rangle_{\mathscr H}  = \langle f(\lambda),g \rangle_E, \quad f \in \mathscr H
  \eeqn
 \end{minipage}
   \right\} 
\end{align}
(refer to \cite[Remark 2.65]{AM} for details).
Conversely, any Hilbert space $\mathscr H$  of $E$-valued functions on a set $X$ satisfying conditions $(\mathsf C1)$ and $(\mathsf C2)$ can be shown to be a reproducing kernel Hilbert space associated with a weak $B(E)$-valued kernel $\kappa_{\mathscr H}$ on $X$ (see \cite[Theorem 2.60]{AM}). 
However, if $\mathscr H$ contains $E,$ then
one can ensure that
any weak $B(E)$-valued kernel is indeed a $B(E)$-valued kernel. Indeed, if $\kappa(\lambda, \lambda)g =0$ for some $\lambda \in \Omega$ and $g \in E,$ then
\beq
\label{+ve}
\|\kappa(\cdot,\lambda)g\|^2 = \langle \kappa(\cdot,\lambda)g, \kappa(\cdot,\lambda)g \rangle_{\mathscr H}  \overset{\eqref{rp}}= 
\langle \kappa(\lambda,\lambda)g, g \rangle_{E},
\eeq
which implies that $\kappa(\cdot,\lambda)g=0$, and hence by \eqref{rp} applied to the constant function $f=g$, we get $g=0.$ 


A bounded open connected subset $\Omega$ of $\mathbb C^d$ is said to be an {\it admissible domain} if it has the following property: For any bounded holomorphic function $\phi : \Omega \rar \mathbb C$,
there exists a sequence $\{p_{n}\}_{n=1}^\infty$ of polynomials such that 
\begin{enumerate}
\item[$\bullet$] for some $M > 0,$ $\|p_{n}\|_{\infty, \Omega}  \Le  M \|\phi\|_{\infty, \Omega}$ for every integer $n \Ge 1$,
\item[$\bullet$] $p_{n}(w)$ converges to $\phi(w)$ as $n \rar \infty$ for every $w \in \Omega$.  
\end{enumerate}
It is well-known that if a bounded domain $\Omega$ has polynomially convex closure in $\mathbb C^d$, then $\Omega$ is admissible provided it is star-shaped or strictly pseudoconvex with $\mathcal C^{2}$ boundary (see, for instance, \cite[Proof of Theorem 4]{Mc}, \cite[Lemma 2.2]{AP}). 
In what follows, we also need the notion of  vector-valued holomorphic function $f : \Omega \rar Z$, where $\Omega$ is a domain in $\mathbb C^d$ and $Z$ is a normed linear space. Recall that $f$ is {\it holomorphic} if $ \phi \circ f$ is holomorphic for every bounded linear functional $\phi$ on $Z$.

Although the following is not standard, we find it convenient for our purpose.
\begin{definition}
 A 
{\it functional Hilbert space} is the quadruple $(\mathscr H, \kappa_{\mathscr H}, \Omega, E)$, where $\Omega$ is an admissible domain in $\mathbb C^d$, $E$ is a separable Hilbert space and $\mathscr H$ is a reproducing kernel Hilbert space of holomorphic functions $f : \Omega \rar E$ associated with the $B(E)$-valued kernel $\kappa_{\mathscr H} : \Omega \times \Omega \rar B(E)$ satisfying the following: 
\begin{itemize}
\item[$\diamond$]($z$-invariance) for any $f \in \mathscr H,$ the function $z_jf : \Omega \rar E$ given by $$(z_jf)(w)=w_jf(w),\quad w=(w_1, \ldots, w_d) \in \Omega$$ belongs to $\mathscr H$ for every $j=1, \ldots, d,$
\item[$\diamond$](Density of polynomials) the space of $E$-valued polynomials in $z_1, \ldots, z_d$ forms a dense subspace of $\mathscr H:$
$\bigvee \{z^\alpha g : \alpha \in \mathbb N^d, ~g \in E\} = \mathscr H.$
\end{itemize}
\end{definition}
\begin{remark} \label{rmk-poly}
Suppose that $\Omega$ contains  the origin $0$.
Then the condition 
\beq \label{nc}
\kappa(\lambda,0)=I_E, \quad \lambda \in \Omega,
\eeq
together with the reproducing property  
implies that $\mathscr H$ contains the space $\mathscr E$ of all $E$-valued constant functions. This fact combined with the $z$-invariance of $\mathscr H$ implies that indeed $\mathscr H$ contains the subspace $\mathscr P$ of all $E$-valued polynomials in $z_1, \ldots, z_d.$ 
Further,
the condition \eqref{rp}
allows to rephrase the normalization condition as 
\beqn \langle f - f(0), g \rangle_{\mathscr H}=0, \quad g \in E, ~f \in \mathscr H. \eeqn
\end{remark}

Let $(\mathscr H, \kappa_{\mathscr H}, \Omega, E)$ be a functional Hilbert space.
A function $\Phi : \Omega \rar B(E)$ is said to be a {\it multiplier} of $\mathscr H$ if 
$\Phi$ is holomorphic and 
\beqn
\Phi f \in \mathscr H ~\mbox{whenever~}f \in \mathscr H,
\eeqn
where $(\Phi f)(w) = \Phi(w)f(w)$ for $w \in \Omega$. 
Any multiplier $\Phi$ induces the bounded linear operator $\mathscr M_{\Phi} : \mathscr H \rar \mathscr H$ given by
\beqn
\mathscr M_{\Phi}f = \Phi f, \quad f \in \mathscr H.
\eeqn
Indeed, in view of the closed graph theorem, this is immediate from 
\beqn
\inp{\Phi f}{\kappa_{\mathscr H}(\cdot, w)h}_{\mathscr H} \overset{\eqref{rp}}= \inp{\Phi(w)f(w)}{h}_E, \quad f \in \mathscr H, ~w \in \Omega, ~h \in E.
\eeqn
By the {\it multiplier norm of $\Phi$}, we understand the operator norm of $\mathscr M_{\Phi}$.
We say that $\Phi : \Omega \rar B(E)$ is {\it bounded} if $\|\Phi\|_{\infty, \Omega} < \infty.$
In this paper, we will be interested in the algebra 
\beqn H^{\infty}_{_{B(E)}}(\Omega) := \Big\{\Phi : \Omega \rar B(E)  ~|~ \Phi ~\mbox{is a bounded holomorphic function}\Big\}.\eeqn
It can be easily deduced from Weierstrass convergence theorem \cite[Chapter I, Theorem 1.9]{R} that $H^{\infty}_{_{B(E)}}(\Omega)$ is a Banach algebra endowed with the sup norm $\|\cdot\|_{\infty, \Omega}$. 
In case $E=\mathbb C$, we use the simpler and standard notation $H^{\infty}(\Omega)$ for $H^{\infty}_{_{B(E)}}(\Omega).$
\begin{remark}
By the definition of the functional Hilbert space, $z_jI_E : \Omega \rar B(E)$ given by 
\beqn
(z_jI_E)(w)=w_jI_E, \quad w = (w_1, \ldots, w_d) \in \Omega,
\eeqn
is a multiplier of $\mathscr H$ for $j=1, \ldots, d.$ In particular, $\mathscr M_{z_jI_E}$ defines a bounded linear operator on $\mathscr H.$ We find it convenient to denote $\mathscr M_{z_jI_E}$ by $\mathscr M_{z_j}.$ Note that the $d$-tuple $\mathscr M_z=(\mathscr M_{z_1}, \ldots, \mathscr M_{z_d})$ is a commuting $d$-tuple in $B(\mathscr H)$.
\end{remark}

We collect below several elementary properties of functional Hilbert spaces and associated multiplication operators.
\begin{proposition} 
\label{eigen}
Let $(\mathscr H, \kappa_{\mathscr H}, \Omega, E)$ be a functional Hilbert space and let $\mathscr M_z$ denote the commuting $d$-tuple of multiplication operators $\mathscr M_{z_1}, \ldots, \mathscr M_{z_d}$ in $B(\mathscr H)$. Then, we have
\begin{enumerate}
\item[(i)] $\bigvee \{\kappa_{\mathscr H} (\cdot,w)g : w \in \Omega, g \in E\}=\mathscr H,$ 
\item[(ii)] $
\kappa_{\mathscr H}(\cdot,w)(g+h)=\kappa_{\mathscr H}(\cdot,w)g + \kappa_{\mathscr H}(\cdot,w)h$ for every $w \in \Omega$ and $g, h \in E,$
\item[(iii)] for any linearly independent subset $G$ of $E$,  $\{\kappa_{\mathscr H}(\cdot, w)g : g \in G\}$ is linearly independent in $\mathscr H$ for every $w \in \Omega$,
\item[(iv)] 
$\kappa_{\mathscr H}(\cdot, 
w)E \subseteq \ker(\mathscr M^*_z - \overline{w})$ for every $w \in \Omega.$
\end{enumerate}
In addition, if $E$ is finite dimensional, then
for any $w \in \Omega,$ we have
\begin{enumerate}
\item[(v)] $\kappa_{\mathscr H} (w,w) \in B(E)$ is invertible such that
$\|\kappa_{\mathscr H} (w,w)^{-1}\| \|\kappa_{\mathscr H} (w,w)\| \Ge 1$, 
\item[(vi)] $\|\kappa_{\mathscr H} (w,w)^{-1}\| \|\kappa_{\mathscr H} (w,w)\| =1$  if and only if $\kappa_{\mathscr H} (w,w)=\mu(w) I_{E}$ for some scalar $\mu(w) >0$,
\item[(vii)] 
$\ker(\mathscr M^*_z - \overline{w})=\kappa_{\mathscr H}(\cdot, 
w)E$, 
\item[(viii)] $\dim \kappa_{\mathscr H}(\cdot, 
w)E= \dim E.$
\end{enumerate}
\end{proposition}
\begin{proof}
Let $w \in \Omega.$ The part (i) follows from the fact that any $f \in \mathscr H$ orthogonal to $\kappa_{\mathscr H}(\cdot, w)E$ satisfies $f(w)=0$ in view of \eqref{rp}.
Part (ii) follows from \eqref{rp} and additivity of the inner-product. To see (iii), in view of (ii), we may suppose that 
$\kappa_{\mathscr H} (\cdot,w)g = 0$ for some $g \in E$. Thus $\kappa_{\mathscr H}(w, w)g=0$, and hence by the injectivity of $\kappa_{\mathscr H}(w, w)$ (see the discussion following \eqref{rp}), $g=0$. This completes the verification of (iii). 

Assume that $E$ is finite dimensional, and let $w \in \Omega.$
Since $\kappa_{\mathscr H} (w,w)$ is injective  and
$E$ is finite dimensional, $\kappa_{\mathscr H} (w,w) \in B(E)$ is invertible.
 The remaining part in (v) is now obvious.
 To see (vi),
note that by \eqref{+ve}, $\kappa_{\mathscr H} (w,w)$ is a positive operator. 
By the spectral theorem, $\kappa_{\mathscr H} (w,w)$ is unitarily equivalent to a diagonal matrix with positive diagonal entries, say, $\mu_j(w)$, $j = 1, \ldots, \dim E$. Further, 
\beqn
\|\kappa_{\mathscr H} (w,w)\| = \max_{1 \Le j \Le \dim E} \mu_j(w), \quad  \|\kappa_{\mathscr H} (w,w)^{-1}\| =\Big( {\displaystyle\min_{1 \Le j \Le \dim E} \mu_j(w)}\Big)^{-1}. 
\eeqn
It follows that $\|\kappa_{\mathscr H} (w,w)^{-1}\| \|\kappa_{\mathscr H} (w,w)\| =1$ if and only if  
\beqn
 {\displaystyle \max_{1 \Le j \Le \dim E} \mu_j(w)}={\displaystyle\min_{1 \Le j \Le \dim E} \mu_j(w)},
\eeqn
which is possible if and only if $\mu_1(w) = \cdots = \mu_{\dim E}(w)$. In this case, $\kappa_{\mathscr H} (w,w)$ must be a scalar multiple of $I_E$. This completes the verification of (vi). 

The facts (iv), (vii) and (viii) may be deduced from \eqref{rp} and the density of $E$-valued polynomials in $z_1, \ldots, z_d$. Indeed, these parts have been noted implicitly in the proof of \cite[Corollaries 4.1.11 and 4.2.11]{CPT}. 
\end{proof}
\begin{remark}
Note that $\kappa_{\mathscr H}(\cdot, w)g$ is precisely the value of the adjoint of  the evaluation map $E_w : \mathscr H \rar E$ evaluated at $g$ as discussed in \cite{PR}, \cite{C-S}.
\end{remark}



\begin{corollary}
Let $(\mathscr H, \kappa_{\mathscr H}, \Omega, E)$ be a functional Hilbert space and let $\mathscr M_z$ denote the commuting $d$-tuple of multiplication operators $\mathscr M_{z_1}, \ldots, \mathscr M_{z_d}$ in $B(\mathscr H)$.
Then the Harte spectrum $\sigma_H(\mathscr M_z)$ of $\mathscr M_z$ contains the closure of $\Omega.$ 
\end{corollary}
\begin{proof}
By Proposition \ref{eigen}(iv), $\{w \in \mathbb C^d : \overline{w} \in \Omega\}$ is contained in the joint point spectrum $\sigma_p(\mathscr M^*_z)$ of $\mathscr M^*_z.$ However, by the general theory \cite{Cu},
\beqn
\sigma_p(\mathscr M^*_z) \subseteq \sigma_H(\mathscr M^*_z) =\{w \in \mathbb C^d : \overline{w} \in \sigma_H(\mathscr M_z)\}.
\eeqn
It follows that ${\Omega} \subseteq \sigma_H(\mathscr M_z).$ The desired conclusion now follows from the fact that the Harte spectrum is closed \cite{Cu}.
\end{proof}
\begin{remark} \label{dominating}
Note that the Harte spectrum $\sigma_H(\mathscr M_z)$ of $\mathscr M_z$ is {\it dominating for the algebra $H^{\infty}(\Omega)$}  in the following sense: \beqn \|f\|_{\infty, \Omega} = \|f\|_{\infty, \sigma_H(\mathscr M_z)\, \cap\, \Omega}, \quad f \in H^{\infty}(\Omega).\eeqn
The condition that a spectral system (e.g. Harte spectrum, essential Taylor spectrum, essential Harte spectrum) is dominating for the algebra of bounded holomorphic functions appears in a variety of results on the invariant subspaces or reflexivity of tuples (see \cite[Corollary 3.7]{Es1}, \cite[Theorem 10.2.2]{Pt}, \cite[Theorem 4.2]{Es4}). 
\end{remark}

We find it convenient to introduce the following terminologies, which resemble with that of {\it von Neumann $d$-tuple} (refer to \cite{Co}, \cite{Di}).
\begin{definition}
Let $(\mathscr H, \kappa_{\mathscr H}, \Omega, E)$ be a functional Hilbert space and let $\mathscr M_z$ denote the commuting $d$-tuple of multiplication operators $\mathscr M_{z_1}, \ldots, \mathscr M_{z_d}$ in $B(\mathscr H)$. We say that {\it $\mathscr M_z$ satisfies von Neumann's inequality} if there exists a constant $K >0$ such that
\beqn
\|p(\mathscr M_z)\|_{B(\mathscr H)} \Le K\, \|p\|_{\infty, \Omega}, \quad p \in \mathbb C[z_1, \ldots, z_d].
\eeqn
We say that {\it $\mathscr M_z$ satisfies the matrix-valued von Neumann's inequality} if there exists a constant $K >0$ such that
\beqn
\|(p_{i, j}(\mathscr M_z))_{_{1 \Le i, j \Le m}}\|_{B(\mathscr H^{(m)})} \Le K\, \|(p_{i, j})_{_{1 \Le i, j \Le m}}\|_{\infty, \Omega}, \quad p_{i, j} \in \mathbb C[z_1, \ldots, z_d], ~m \in \mathbb N.
\eeqn
\end{definition}
\begin{remark} \label{rmk-pb}
Note that the multiplication tuple $\mathscr M_z$ satisfies von Neumann's inequality (resp. matrix-valued von Neumann's inequality) if the polynomial functional calculus $\Phi(p)=p(\mathscr M_z)~(p \in \mathbb C[z_1, \ldots, z_d])$ is {\it bounded} (resp. {\it  completely bounded}) in the sense of \cite{P} and \cite{Pi}.  
\end{remark}

We record the following known fact for ready reference 
(see \cite[Corollary 7.7]{P} or \cite[Theorem 1.2.2]{Ar} and the remark following it).
\begin{lemma} \label{dilation-von}
Let $\Omega$ be a bounded domain in $\mathbb C^d$ and let $T$ be a commuting $d$-tuple in $B(\mathcal H)$. Suppose there exists a commuting $d$-tuple $N$ of normal operators in $B(\mathcal K)$ for some Hilbert space $\mathcal K$ containing  $\mathcal H$ such that 
\beqn \sigma(N) \subseteq \overline \Omega ~\mbox{and~}
p(T)=P_{\mathcal H}p(N)|_{\mathcal H}, \quad p \in \mathbb C[z_1, \ldots, z_d].
\eeqn
Then 
$T$ satisfies 
\beqn
\|(p_{i, j}(T))_{_{1 \Le i, j \Le m}}\|_{B(\mathscr H^{(m)})} \Le  \|(p_{i, j})_{_{1 \Le i, j \Le m}}\|_{\infty, \Omega}, \quad p_{i, j} \in \mathbb C[z_1, \ldots, z_d], ~m \in \mathbb N.
\eeqn
\end{lemma}

\begin{List} \label{list}
We list here some known cases \cite[Chapter I, Theorems 4.1]{SF}, \cite[Pg 88]{Ando}, \cite[Theorem 1.1]{Hz} \cite[Pg 987]{MV}, \cite[Proposition 2]{AL} in which the pair $(\Omega, T)$ satisfies the hypothesis of 
Lemma \ref{dilation-von}:
\begin{enumerate}
\item[$\bullet$] $\Omega=\mathbb D$, $T$ is any contraction,
 \item[$\bullet$] $\Omega=\mathbb D^2$, $T$ is a  contractive $2$-tuple,
\item[$\bullet$] $\Omega=\mathbb D^d$, $T$ is a contractive $d$-variable  weighted shift with positive weights,
 \item[$\bullet$] $\Omega=\mathbb D^d$, $T$ is a joint contractive $d$-tuple.
 \item[$\bullet$] $\Omega=\mathbb B^d$, $T$ is a commuting $d$-tuple such that
 \beqn
 \sum_{j=0}^k (-1)^j {k \choose j} \sum_{\underset{\tiny{|\alpha|=j}}{\alpha \in \mathbb N^d}} \frac{j!}{\alpha!} T^{*\alpha}T^{\alpha} \Ge 0, \quad k=1, \ldots, d. 
 \eeqn
\end{enumerate}
\end{List}
Under some natural assumptions on the $B(E)$-valued kernels $\kappa_{\mathscr H}$ (see \eqref{assumption} and \eqref{g-rate}), the commutants of multiplication tuples $T=\mathscr M_z$ in $B(\mathscr H)$, falling in any one of the classes mentioned in List \ref{list}, can be identified with the algebra $H^{\infty}_{_{B(E)}}(\Omega)$ of $B(E)$-valued bounded holomorphic functions on $\Omega$ (see Corollary \ref{coro-list}). We will also show that $\mathscr M_z$ is reflexive in all the above cases (see Corollary \ref{coro-list2}).

\section{Commutants}

The first main result of this paper identifies commutants of multiplication tuples $\mathscr M_z$ on certain functional Hilbert spaces. A special case of this result (under the additional assumption that the joint kernel of $\mathscr M^*_z-\overline{\lambda}$ is $1$-dimensional for every $\lambda \in \Omega$) has been essentially obtained in \cite[Theorem 5.2]{C-S}.
\begin{theorem} \label{commutant}
Let $(\mathscr H, \kappa_{\mathscr H}, \Omega, E)$ be a functional Hilbert space with finite dimensional $E$ and let $\mathscr M_z$ denote the commuting $d$-tuple of multiplication operators $\mathscr M_{z_1}, \ldots, \mathscr M_{z_d}$ in $B(\mathscr H)$.
Suppose that the reproducing kernel $\kappa_{\mathscr H}$ satisfies 
the following conditions:
\beq \label{assumption}
\inp{f(\cdot)}{g}_{_E} h \in \mathscr H ~\mbox{for every}~g, h \in E~\mbox{and~}f \in \mathscr H, 
\eeq
\beq \label{g-rate}
\sup_{w \in \Omega} \displaystyle \|\kappa_{\mathscr H}(w, w)\|_{_{B(E)}} \|\kappa_{\mathscr H}(w, w)^{-1}\|_{_{B(E)}} < \infty
\eeq
$($see Proposition \ref{eigen}$($v$))$.
Then the 
commutant $\{{\mathscr M}_z\}'$ of ${\mathscr M}_z$ is isometrically isomorphic to the Banach algebra $$\mathscr R:=\{\Phi \in H^{\infty}_{_{B(E)}}(\Omega) : \mathscr M_{\Phi} \in \{\mathscr M_z\}'\}$$ 
endowed with the multiplier norm $\|\cdot\|_{B(\mathscr H)}$.
In addition, if the multiplication $d$-tuple $\mathscr M_z$ satisfies the matrix-valued von Neumann's inequality, then $\{{\mathscr M}_z\}'$ is isometrically isomorphic to the Banach algebra $(H^{\infty}_{_{B(E)}}(\Omega), ~\|\cdot\|_{B(\mathscr H)})$.
\end{theorem}
\begin{remark} \label{rmk-c}
The conditions \eqref{assumption} and \eqref{g-rate} are natural in the following sense: 
\begin{enumerate}
\item[$\bullet$] Any $E$-valued polynomial $f$ satisfies \eqref{assumption}. \item[$\bullet$] In case $\dim E=1$, $\kappa_{\mathscr H}$ satisfies \eqref{assumption} as well as \eqref{g-rate}.
\end{enumerate}
\end{remark}
\begin{proof}
Let $T \in B(\mathscr H)$ be in the commutant 
of ${\mathscr M}_z$ and let $w=(w_1, \ldots, w_d)$ be in $\Omega$. 
By Proposition \ref{eigen}(vii), for any $g \in E$ and $j=1, \ldots, d$, $$(\mathscr M^*_{z_j} - \overline{w}_j)T^* \kappa_{\mathscr H}(\cdot, 
w)g = T^*(\mathscr M^*_{z_j} - \overline{w}_j)\kappa_{\mathscr H}(\cdot, w)g =0.$$ 
Thus $T^* $ maps $\kappa_{\mathscr H}(\cdot, 
w)E$ into itself. 
Consequently, there exists an operator $\Phi(w)$ in $B(E)$ such that 
\beq
\label{multiplier}
T^*\kappa_{\mathscr H}(\cdot, 
w)g =\kappa_{\mathscr H}(\cdot, 
w)\Phi(w)^*g, \quad g \in E
\eeq
(see, for instance, \cite[Pg 32]{AM}).
Note that for any $h \in \mathscr H$ and $g \in E$,
\beqn
\inp{(Th)(w)}{g}_E &\overset{\eqref{rp}}=&
 \inp{Th}{\kappa_{\mathscr H}(\cdot, w)g}_{\mathscr H} 
 =  \inp{h}{T^*\kappa_{\mathscr H}(\cdot, w)g}_{\mathscr H} \\ &\overset{\eqref{multiplier}}=&\inp{h}{\kappa_{\mathscr H}(\cdot, w)\Phi(w)^*g}_{\mathscr H} \overset{\eqref{rp}}= \inp{\Phi(w)h(w)}{g}_E.
\eeqn
This shows that $T=\mathscr M_{\Phi}$ for a map $\Phi : \Omega \rar B(E)$.
Further,  since $Tg$ is holomorphic for every $g \in E,$ so is $\Phi$.
We claim that
\beq \label{claim}
\sup_{w \in \Omega}\frac{\|\Phi(w)\kappa_{\mathscr H}(w, w)\|_{_{B(E)}}}{\|\kappa_{\mathscr H}(w, w)\|_{_{B(E)}}} \Le \|\mathscr M_{\Phi}\|.
\eeq
Indeed, for any $w \in \Omega$ and unit vectors $g, \tilde g \in E,$
\beqn
|\inp{\Phi(w)\kappa_{\mathscr H}(w, w)g}{\tilde g}_E| &\overset{\eqref{rp}}=& |\inp{M_\Phi\kappa_{\mathscr H}(\cdot, w)g}{\kappa_{\mathscr H}(\cdot, w)\tilde g}_{\mathscr H}| \\ & \Le & \|\mathscr M_{\Phi}\| \|\kappa_{\mathscr H}(\cdot, w)g\|_{\mathscr H} \|\kappa_{\mathscr H}(\cdot, w)\tilde g\|_{\mathscr H} \\
& \overset{\eqref{+ve}}\Le & \|\mathscr M_{\Phi}\| \|\kappa_{\mathscr H}(w, w)\|_{_{B(E)}}.
\eeqn
Taking supremum over all unit vectors $g, \tilde g \in E$,  the claim stands verified.
Combining \eqref{claim} with the assumption \eqref{g-rate}, we obtain for any $w \in \Omega$,
\beq
\label{norm-e}
\|\Phi(w)\|_{_{B(E)}} & \Le & \|\Phi(w)\kappa_{\mathscr H}(w, w)\|_{_{B(E)}}  \|\kappa_{\mathscr H}(w, w)^{-1}\|_{_{B(E)}} \notag \\ & \Le & \|\mathscr M_{\Phi\,}\| \sup_{w \in \Omega} \|\kappa_{\mathscr H}(w, w)\|_{_{B(E)}} \|\kappa_{\mathscr H}(w, w)^{-1}\|_{_{B(E)}},
\eeq
and hence $T=\mathscr M_{\Phi}$ for $\Phi$ in the algebra $H^{\infty}_{_{B(E)}}(\Omega).$ 
Clearly, $(\mathscr R, ~\|\cdot\|_{B(\mathscr H)})$ is a Banach algebra. 
It follows that the mapping $\mathscr F : \{\mathscr M_z\}' \rar (\mathscr R, ~\|\cdot\|_{B(\mathscr H)})$ given by  $\mathscr F(\mathscr M_{\Phi}) = \Phi$ is an isometric isomorphism. 

To see the remaining part, assume that $\mathscr M_z$ satisfies the matrix-valued von Neumann's inequality. It suffices to check that for every bounded holomorphic function $\Phi : \Omega \rar B(E)$ is a multiplier of $\mathscr H.$ To see that, let $\Phi : \Omega \rar B(E)$ be a bounded holomorphic function. Let $m:=\dim E $ and $\mathcal B:=\{g_j : j =1, \ldots,  m\}$ be an orthonormal basis of $E$. For $w \in \Omega$, let $(\phi_{i, j}(w))_{1 \Le i, j \Le m}$ be the matrix representation of $\Phi(w)$ with respect to the basis $\mathcal B$ of $E$. Fix $ i,j =1, \ldots, m.$ Since $\Phi$ is bounded holomorphic, so is $\phi_{i, j}$. 
By assumption, $\Omega$ is an admissible domain in $\mathbb C^d,$ and hence there exists a sequence $\{p_{i,j}^{(n)}\}_{n=1}^\infty \subseteq \mathbb C[z_1, \ldots, z_d]$ such that 
\begin{enumerate}
\item[$(\mathsf P1)$] ~$\|p_{i,j}^{(n)}\|_{\infty, \Omega}  \Le   M\|\phi_{i, j}\|_{\infty, \Omega}$ for some constant $M >0$, 
\item[$(\mathsf P2)$] ~$p_{i,j}^{(n)}(w)$ converges to $\phi_{i, j}(w)$ as $n \rar \infty$ for every $w \in \Omega$.  
\end{enumerate}
Let $M_m(\mathbb C)$ denote the Banach algebra of $m \times m$ matrices of complex entries endowed with the operator norm, and recall
the fact that \beq \label{matrix-norm} \|(a_{i,j})\|_{M_m(\mathbb C)} \Le m \max_{1 \Le i, j \Le m} |a_{i, j}|, \quad (a_{i, j}) \in M_m(\mathbb C).\eeq
Since ${\mathscr M}_z$ satisfies the matrix-valued von Neumann's inequality, 
for some constant $K>0$, we obtain
\beqn
\big\|\big(p_{i,j}^{(n)}({\mathscr M}_z)\big)\big\|_{B(\mathscr H^{(m)})} &\Le &  K\, \sup_{z \in {\Omega}}\big\|\big(p_{i,j}^{(n)}(z)\big)\big\|_{M_m(\mathbb C)} \\
&\overset{\eqref{matrix-norm}} \Le & 
 K'\, \max_{1 \Le i,j \Le m} \|p_{i,j}^{(n)}\|_{\infty, \Omega} \\
& \overset{(\mathsf P1)}  \Le & ~K'\, M\max_{1 \Le i,j \Le m} \|\phi_{i,j}\|_{\infty, \Omega},
\eeqn
where $K'=m\,K.$
Thus for any $F \in \mathscr H^{(m)}$,
$\big\{\big(p_{i,j}^{(n)}({\mathscr M}_z)\big)F\big\}_{n=1}^\infty$ is a  bounded sequence in $\mathscr H^{(m)}$. By \cite[Theorem 3.6.11]{Si},  $\big\{\big(p_{i,j}^{(n)}({\mathscr M}_z)\big)F\big\}_{n=1}^\infty$ admits a weakly convergent subsequence. For simplicity, we assume that
$\big\{\big(p_{i,j}^{(n)}({\mathscr M}_z)\big)F\big\}_{n=1}^\infty$ itself converges weakly to, say, $\tilde F \in \mathscr H^{(m)}$. That is, 
\beq \label{w-cgn} \lim_{n \rar \infty} \big \langle{\big(p_{i,j}^{(n)}({\mathscr M}_z)\big)F},{H}\big \rangle = \inp{\tilde F}{H}~\mbox{for all} ~H \in \mathscr H^{(m)}. \eeq
Let $F = \oplus_{i=1}^m f_i \in \mathscr H^{(m)}$  and write $\tilde F = \oplus_{i=1}^m \tilde f_i$. 
Fix an integer $k = 1, \ldots, m,$ $g \in E$, $w \in \Omega,$ and set $H := \oplus_{j=1}^m h_j$, where 
\beqn
h_j= \begin{cases} \kappa_{\mathscr H}(\cdot,w)g & \mbox{if} ~j = k, 
\\  0 & \mbox{otherwise}. 
\end{cases}
\eeqn
Note that
\beqn
\big \langle{\big(p_{i,j}^{(n)}({\mathscr M}_z)\big)F},{H}\big \rangle_{\mathscr H^{(m)}}  &=& \Big \langle{\oplus_{j=1}^m \Big(\sum_{i=1}^m p_{j,i}^{(n)}({\mathscr M}_z)f_i\Big)},{\oplus_{j=1}^m h_j} \Big \rangle_{\mathscr H^{(m)}} \\
&=& \sum_{i=1}^m \big \langle{ p_{k,i}^{(n)}({\mathscr M}_z)f_i},{\kappa_{\mathscr H}(\cdot,w)g} \big \rangle_{\mathscr H}\\
&=& \sum_{i=1}^m p_{k,i}^{(n)}(w) \big \langle{f_i(w)},{g}\big \rangle_E.
\eeqn 
Since $p_{i,j}^{(n)}(w)$ converges to $\phi_{i, j}(w)$ (see $(\mathsf P2)$),
after letting $n \rar \infty$ on both sides, by \eqref{w-cgn}, we obtain
$$\inp{\tilde F}{H}_{\mathscr H^{(m)}} =\sum_{i=1}^m \phi_{k,i}(w) \inp{f_i(w)}{g}_E.$$
However, $\inp{\tilde F}{H}_{\mathscr H^{(m)}} =\inp{\tilde f_k(w)}{g}_E,$ so that
$$\inp{\tilde f_k(w)}{g}_E=\sum_{i=1}^m \phi_{k,i}(w) \inp{f_i(w)}{g}_E~\mbox{for every~}g \in E.$$
Consequently, \beq \label{9} \tilde f_k(w)= \sum_{i=1}^m \phi_{k,i}(w)f_i(w), \quad  w \in \Omega, ~k=1, \ldots, m. \eeq 
The preceding discussion shows that for every $F = \oplus_{i=1}^m f_i \in \mathscr H^{(m)}$, we obtain $\tilde{F} = \oplus_{i=1}^m \tilde{f}_i \in \mathscr H^{(m)}$ governed by \eqref{9}.
We apply this association to $F^{(k)}=\oplus_{i=1}^m f^{(k)}_i,$ $k=1,\ldots, m$ given by
\beq
\label{f-i}
f^{(k)}_i(w):={\inp{f(w)}{g_i}}_{E}\, g_k, \quad w \in \Omega, ~i=1, \ldots, m,
\eeq
where $f \in \mathscr H$. By assumption \eqref{assumption}, 
$f^{(k)}_i \in \mathscr H$ for all $i=1, \cdots, m$.
Hence the above association yields $\tilde F^{(k)}=\oplus_{i=1}^m \tilde f^{(k)}_i \in \mathscr H^{(m)},$ $k=1,\ldots, m$.
It follows from \eqref{9} that 
 \beqn \tilde f^{(k)}_j(w) =
  \sum_{i=1}^m \phi_{j,i}(w)f^{(k)}_i(w), \quad  w \in \Omega, ~j=1, \ldots, m, \eeqn 
and hence for any $w \in \Omega,$
\beqn \sum_{k=1}^m \tilde f^{(k)}_k(w) &=& \sum_{k=1}^m  \sum_{i=1}^m \phi_{k,i}(w)f^{(k)}_i(w) \\
&\overset{\eqref{f-i}}=& \sum_{k=1}^m \sum_{i=1}^m \phi_{k,i}(w){\inp{f(w)}{g_i}}_E\, g_k 
= (\mathscr M_{\Phi}f)(w). \eeqn 
Since $\tilde f^{(k)}_k \in \mathscr H$ for $k=1, \ldots, m$, $\mathscr M_{\Phi}f \in \mathscr H$, and hence $\Phi$ is a multiplier of $\mathscr H.$ Trivially, $\mathscr M_{\Phi}$ commutes with $\mathscr M_z.$
\end{proof}
\begin{remark} \label{Hinfty}
It is evident from the proof that the matrix-valued von Neumann's inequality is required only for the choice $m=\dim E$. Further, it has been pointed out by the anonymous referee that one can renorm $\mathscr H$ (assuming \eqref{assumption}), so that it is a RKHS in this equivalent norm
with kernel of the form $\kappa(z, w)I_E$, where $\kappa$ is a scalar-valued kernel. This can be achieved by endowing $\mathscr K_j=\{\inp{f(\cdot)}{e_j} : f \in \mathscr H\}$ with the norm, which makes $f \mapsto  \inp{f(\cdot)}{e_j}$ a quotient map, where $\{e_1, \ldots, e_{\dim E}\}$ is an orthonormal basis. Note that \eqref{assumption} allows us to identify, up to similarity, the $\mathbb C[z_1, \ldots, z_d]$-Hilbert modules $\mathscr H$ and $\oplus_{j=1}^{\dim E}\mathscr K_j$. In particular, Theorem \ref{commutant} can be recovered from its scalar-valued counter-part. 
The advantage gained in this process is that the assumption \eqref{g-rate} can be relaxed. 
\end{remark}

In the remaining part of this section, we discuss several applications of Theorem \ref{commutant}.
The first of which is a direct consequence of Theorem \ref{commutant} and Lemma \ref{dilation-von}.
\begin{corollary} \label{coro-list}
Let $(\mathscr H, \kappa_{\mathscr H}, \Omega, E)$ be a functional Hilbert space with finite dimensional $E$ and let $\mathscr M_z$ denote the commuting $d$-tuple of multiplication operators $\mathscr M_{z_1}, \ldots, \mathscr M_{z_d}$ in $B(\mathscr H)$.
Suppose that the reproducing kernel $\kappa_{\mathscr H}$ satisfies \eqref{assumption} and \eqref{g-rate}.
If the pair
$(\Omega, \mathscr M_z)$ falls in the List \ref{list}, then the
commutant $\{{\mathscr M}_z\}'$ of ${\mathscr M}_z$ is equal to the algebra $H^{\infty}_{_{B(E)}}(\Omega).$
\end{corollary}
\begin{remark}
Assume that $\dim E=1$. Then, by Remark \ref{rmk-c}, $\kappa_{\mathscr H}$ satisfies \eqref{assumption} and \eqref{g-rate}. Thus the above corollary is applicable to any pair $(\Omega, ~T:=\mathscr M_z)$ falling in the List \ref{list}.
\end{remark}

The following answers when the commutant of the multiplication tuple $\mathscr M_z$ is abelian (cf. \cite[Theorem 4.12]{SF}, \cite[Section 4, Corollary 2]{S}, \cite[Theorem 7]{Y-2}). 
\begin{corollary} \label{comm-abelian}
Under the assumptions of Theorem \ref{commutant}, the following statements are equivalent:
\begin{enumerate}
 \item[(i)] The commutant $\{\mathscr M_z\}'$ of $\mathscr M_z$ is abelian. 
 \item[(ii)] $\mathscr M_z$ is irreducible.
\item[(iii)]  $\dim E =1.$
\end{enumerate}
\end{corollary}
\begin{proof}
In the proof, we need the following properties of multipliers:
\begin{enumerate}
 \item[(a)] $\mathscr M_{\Phi}\mathscr M_{\Psi} =\mathscr M_{\Psi}\mathscr M_{\Phi}$ if and only if 
${\Phi}{\Psi} ={\Psi}{\Phi}.$
\item[(b)] $\mathscr M_{\Phi}$ is an orthogonal 
projection if and only if so is $\Phi.$ 
\end{enumerate}
Suppose that (i) holds. If $S, T \in \{{\mathscr M}_z\}'$ then by 
the preceding theorem, $S=\mathscr M_{\Phi}$ and $T=\mathscr M_{\Psi}$ for some bounded holomorphic $B(E)$-valued functions $\Phi, 
\Psi$ on $\Omega$. But then by (a), we must 
have $\Phi \Psi = \Psi \Phi.$ Clearly, in case $\dim E > 1,$ there are constant (and hence bounded and holomorphic) functions $\Phi, \Psi$ which do not commute. Thus (i) holds if and only if $\dim E=1.$ This proves 
the equivalence of (i) and (iii). 

Suppose that $\dim E=1.$ Let $P$ be an orthogonal projection belonging to 
$\{\mathscr M_z\}'.$  By 
Theorem \ref{commutant}, $P=\mathscr M_\Phi$ for some bounded holomorphic $\Phi : \Omega \rar B(E)$. By (b), $\Phi^2=\Phi$, and hence either $\Phi=0$ or $\Phi=I_E.$ This proves that (iii) $\Rightarrow$ (ii).

Suppose that $\dim E \Ge 2$. Consider the 
constant $B(E)$-valued rank one orthogonal projection $\Phi$. By (b), $\mathscr M_\Phi$ is an orthogonal projection. Further, by (a),  
$\mathscr M_{\Phi}$ belongs to $\{\mathscr M_z\}'$. This shows that $\mathscr M_z$ is 
reducible, and hence (ii) $\Rightarrow$ (iii).
\end{proof}
\begin{remark}
Unlike the case of $\dim E=1$ (see \cite[Theorem 2]{Sh-Wa}), the commutant of $\mathscr M_z$ differs from the WOT-closed algebra $\mathscr W_{\mathscr M_z}$ generated by $\mathscr M_z$ and the identity operator $I_{\mathscr H}$ on $\mathscr H$. Indeed, if $\dim E > 1$, then $\{\mathscr M_z\}'$ is non-abelian, whereas $\mathscr W_{\mathscr M_z}$ is easily seen to be abelian.
\end{remark}

The following identifies the commutant of an orthogonal direct sum of finitely many copies of a contractive multiplication operator on a  reproducing kernel Hilbert space of scalar-valued holomorphic functions.
\begin{corollary}
Let $(\mathscr H, \kappa_{\mathscr H}, \mathbb D, \mathbb C)$ be a functional Hilbert space and let $m$ be a positive integer.
If  the operator $\mathscr M_z$ of multiplication by $z$ is contractive, then
the 
commutant $\{{\mathscr M}^{(m)}_z\}'$ of ${\mathscr M}^{(m)}_z \in B(\mathscr H^{(m)})$ is equal to the algebra $H^{\infty}_{_{B(\mathbb C^m)}}(\mathbb D).$
\end{corollary}
\begin{proof}
Suppose that $\mathscr M_z$ is contractive.
Let $f=\oplus_{j=1}^m f_j \in \mathscr H^{(m)}$ and let $K(z, w)=\kappa_{\mathscr H}(z, w)I_{\mathbb C^m}$.
Then, for any $x=(x_1, \ldots, x_m) \in \mathbb C^m$,
\beqn
\inp{f}{K(\cdot, w)x} = \sum_{j=1}^m \inp{f_j}{k_{\mathscr H}(\cdot, w)x_j}=\sum_{j=1}^m f_j(w)\overline{x}_j=\inp{f(w)}{x}.
\eeqn
In particular, $\mathscr H^{(m)}$ is a reproducing kernel Hilbert space associated with the $B(\mathbb C^m)$-valued kernel $K$ (see \cite[Pg 99-100]{PR}). 
This also shows that $$\inp{f(\cdot)}{x}y  =\sum_{j=1}^m \overline{x}_j  f_j(\cdot) y
\in \mathscr H^{(m)} ~\mbox{for every}~x, y \in \mathbb C^m~\mbox{and~}f \in \mathscr H^{(m)},$$
which is precisely the condition \eqref{assumption}. 
Since $K$ trivially satisfies the boundedness condition \eqref{g-rate}, the desired conclusion is immediate from Theorem \ref{commutant}.
\end{proof}

\section{Reflexivity}

The main result of this section shows that the multiplication tuple $\mathscr M_z$ on any functional Hilbert space $\mathscr H$ satisfying von Neumann's inequality is reflexive. 
Our proof is inspired by the technique usually employed either to compute commutant or to establish reflexivity of the multiplication tuple $\mathscr M_z$ on a reproducing kernel Hilbert space of scalar-valued holomorphic functions (cf.  \cite[Proof of Lemma 5]{Sh-Wa}, \cite[Section 10, Proposition 37]{S},  \cite[Chapter VII, Lemma 8.2]{Co}, and \cite[Section 0, Theorem 4]{Mc}).
The novelty of this refined technique is that it ensures reflexivity of $\mathscr M_z$ under the mild assumption that it satisfies von Neumann's inequality.
\begin{theorem} \label{reflexive}
Let $(\mathscr H, \kappa_{\mathscr H}, \Omega, E)$ be a functional Hilbert space and let $\mathscr M_z$ denote the commuting $d$-tuple of multiplication operators $\mathscr M_{z_1}, \ldots, \mathscr M_{z_d}$ in $B(\mathscr H)$.
Suppose that $\mathscr M_z$ satisfies von Neumann's inequality.
Then $\mathscr M_z$ is reflexive.
\end{theorem}
\begin{proof}
Clearly, $\mathscr W_{\mathscr M_z} \subseteq \mbox{AlgLat}\, \mathscr W_{\mathscr M_z}$. To see the reverse inclusion, let
 $A$ belong to $\mbox{AlgLat}\, \mathscr W_{\mathscr M_z}$, and note that
$\mbox{Lat}\, \mathscr W_{\mathscr M^*_{z}} \subseteq \mbox{Lat}\, {A^*}$. By Proposition \ref{eigen}(iv),
\beqn
\mathcal M_{w, g}:= \{a \kappa_{\mathscr H}(\cdot, w)g : a \in \mathbb C\} \in \mbox{Lat}\, \mathscr W_{\mathscr M^*_{z}}, \quad w \in \Omega, ~g \in E.
\eeqn
Thus there exists a scalar $\phi_{g}(w)$ such that \beq \label{action-A} A^*\kappa_{\mathscr H}(\cdot, w)g=\overline{\phi_{g}(w)}\,\kappa_{\mathscr H}(\cdot, w)g, \quad w \in \Omega, ~g \in E. \eeq Let $\{g_n\}_{n \in \Lambda}$ be an orthonormal basis of $E$. We contend that 
\beq \label{gi-gk}
\phi_{g_j} = \phi_{g_k}, \quad j, k \in \Lambda.
\eeq
For $w \in \Omega$ and $j, k \in \Lambda$, by two applications of Proposition \ref{eigen}(ii), we have
\beqn
\overline{\phi_{g_j}(w)}\,\kappa_{\mathscr H}(\cdot, w)g_j + \overline{\phi_{g_k}(w)}\,\kappa_{\mathscr H}(\cdot, w)g_k & \overset{\eqref{action-A}} = & A^*\kappa_{\mathscr H}(\cdot, w)g_j + A^*\kappa_{\mathscr H}(\cdot, w)g_k \\ &=& A^*\Big(\kappa_{\mathscr H}(\cdot, w)g_j + \kappa_{\mathscr H}(\cdot, w)g_k\Big) \\ &=& A^*\kappa_{\mathscr H}(\cdot, w)(g_j + g_k) \\ &\overset{\eqref{action-A}}=& \overline{\phi_{g_j + g_k}(w)}\,\kappa_{\mathscr H}(\cdot, w)(g_j + g_k) 
\\ &=& \overline{\phi_{g_j + g_k}(w)}\,(\kappa_{\mathscr H}(\cdot, w)g_j + \kappa_{\mathscr H}(\cdot, w)g_k).
\eeqn
However, by Proposition \ref{eigen}(iii), $\{\kappa_{\mathscr H}(\cdot, w)g_j\}_{j \in \Lambda}$ forms a linearly independent subset of $\mathscr H,$ and hence we conclude that
\beqn
\overline{\phi_{g_j}(w)} = \overline{\phi_{g_j + g_k}(w)} =\overline{\phi_{g_k}(w)}. 
\eeqn
This yields \eqref{gi-gk}. Let $\phi : \Omega \rar \mathbb C$ be a function such that $\phi_{g_j} = \phi$ for all $j \in \Lambda.$ It is now immediate from  \eqref{action-A} that
\beq \label{action-A-star} A^*\kappa_{\mathscr H}(\cdot, w)g_j=\overline{\phi(w)}\,\kappa_{\mathscr H}(\cdot, w)g_j, \quad w \in \Omega, ~j \in \Lambda. \eeq
This implies that 
\beqn
\quad |\overline{\phi(w)}|\|\kappa_{\mathscr H}(\cdot, w)g_j\|=\|A^*\kappa_{\mathscr H}(\cdot, w)g_j\| \Le \|A^*\|\|\kappa_{\mathscr H}(\cdot, w)g_j\|, \quad w \in \Omega, ~j \in \Lambda.
\eeqn
Since $\kappa_{\mathscr H}(\cdot, w)g_j \neq 0$ (see Proposition \ref{eigen}(iii)),
the above estimate shows that $\|\phi\|_{\infty, \Omega} \Le \|A^*\|$, and hence $\phi$ is bounded.
Further, for any $f \in \mathscr H$ and $w \in \Omega,$ 
\begin{alignat}{2}
\label{A-phi}
(Af)(w) & ~=~ \sum_{j \in \Lambda} \inp{(Af)(w)}{g_j}g_j  \overset{\eqref{rp}}   =  \sum_{j \in \Lambda} \inp{Af}{\kappa_{\mathscr H}(\cdot, w)g_j}g_j  & \notag \\ 
&\overset{\eqref{action-A-star}}  =  \phi(w) \sum_{j \in \Lambda} \inp{f}{\kappa_{\mathscr H}(\cdot, w)g_j}g_j   \overset{\eqref{rp}}  =  \phi(w)f(w). &
\end{alignat}
Since $Af \in \mathscr H$, $\phi f \in \mathscr H$ for every $f \in \mathscr H.$ This shows that $\phi g_j \in \mathscr H$ for every $j \in \Lambda$. However, $\phi(w)=\inp{\phi(w)g_j}{g_j}_E$, $j \in \Lambda$, and hence $\phi$ is holomorphic. 
Since $\Omega$ is admissible, there exists a sequence of polynomials $\{p_n\}_{n \in \mathbb N} \subseteq \mathbb C[z_1, \ldots, z_d]$ such that for some positive constant $M,$
\beq
\label{p-cgn}
\|p_n\|_{\infty, \Omega} \Le M \|\phi\|_{\infty, \Omega}, ~\lim_{n \rar \infty} p_n(w) = \phi(w), \quad w \in \Omega.
\eeq
Note that for $f \in \mathscr H$, $w \in \Omega$ and $g \in E,$
\beq \label{limit}
\lim_{n \rar \infty} \inp{p_n(\mathscr M_z)f}{\kappa_{\mathscr H}(\cdot, w)g}_{\mathscr H}  &\overset{\eqref{rp}}=& \lim_{n \rar \infty} \inp{p_n(w)f(w)}{g}_E \notag \\ & \overset{\eqref{p-cgn}}=&  \inp{\phi(w)f(w)}{g}_E \notag \\ &\overset{\eqref{A-phi}} =& \inp{(Af)(w)}{g}_E \notag \\ & \overset{\eqref{rp}} = & \inp{Af}{\kappa_{\mathscr H}(\cdot, w)g}_{\mathscr H}.
\eeq
By von Neumann's inequality and \eqref{p-cgn}, $\{p_n(\mathscr M_z)\}_{n \in \mathbb N}$ is a bounded sequence. 
This combined with \eqref{limit} and Proposition \ref{eigen}(i) shows that $\{p_n(\mathscr M_z)\}_{n \in \mathbb N}$ converges to $A$ in WOT. This shows that
$
A \in \mathscr W_{\mathscr M_z},$ and hence
 $\mbox{AlgLat}\, \mathscr W_{\mathscr M_z} \subseteq \mathscr W_{\mathscr M_z}.$ This completes the proof of the theorem.
\end{proof}
\begin{remark} \label{rmk-reflexive}
We note the following:
\begin{enumerate}
\item It is evident from the proof of Theorem \ref{reflexive} that the assumption of the density of $E$-valued analytic polynomials is inessential (cf. Remark \ref{rmk-poly}). 
\item Theorem \ref{reflexive} is applicable to the multiplication tuple $\mathscr M_z$ acting on a functional Hilbert space, which is a {\it $\Gamma$-contraction} in the sense of \cite{AY}.
\item 
As pointed out by the anonymous referee, the part of Theorem \ref{reflexive} till \eqref{A-phi} can also be deduced from \cite[Corollary 2.2]{Ba}. However, after applying it to the algebra $\mathscr S$ of all diagonal operators on $E$, one may conclude that every operator in $\mbox{AlgLat}(\mathscr W_{\mathscr M_z})$ is of the form $\mathscr M_{\phi}$ for a $\mathscr S$-valued multiplier $\phi$. Since $\dim \mathscr S$ could be bigger than $1$, it is not clear to the authors how to deduce that the multiplier $\phi$ is indeed scalar-valued.
\end{enumerate}
\end{remark}

%

We discuss below several consequences of Theorem \ref{reflexive}.
\begin{corollary} \label{coro-list2}
Let $(\mathscr H, \kappa_{\mathscr H}, \Omega, E)$ be a functional Hilbert space  and let $\mathscr M_z$ denote the commuting $d$-tuple of multiplication operators $\mathscr M_{z_1}, \ldots, \mathscr M_{z_d}$ in $B(\mathscr H)$.
If the pair
$(\Omega, \mathscr M_z)$ falls in the List \ref{list}, then $\mathscr M_z$ is reflexive.
\end{corollary}

The next corollary shows that certain joint subnormal multiplication tuples acting on functional Hilbert spaces are reflexive. In particular, it recovers special cases of \cite[Theorem 3]{OT}, \cite[Theorem 2.4]{B1}, \cite[Theorem 3.8]{Es3}, \cite[Theorem 5]{Di-1} and \cite[Theorem 1]{Es}. Recall that a commuting $d$-tuple $S=(S_1, \ldots, S_d)$ in $B(\mathcal H)$ is {\it joint subnormal} if there exist a Hilbert space $\mathcal K \supseteq \mathcal H$ and a commuting $d$-tuple $N$ of normal operators $N_1, \ldots, N_d$ in $B(\mathcal K)$ such that 
$S_j = {N_j}|_{\mathcal H}$ for $j=1, \ldots, d.$
\begin{corollary}
Let $(\mathscr H, \kappa_{\mathscr H}, \Omega, E)$ be a functional Hilbert space and let $\mathscr M_z$ denote the commuting $d$-tuple of multiplication operators $\mathscr M_{z_1}, \ldots, \mathscr M_{z_d}$ in $B(\mathscr H)$. If $\mathscr M_z$ is a joint subnormal $d$-tuple with normal extension $N$ such that $\sigma(N) \subseteq \overline{\Omega}$, then 
$\mathscr M_z$ is reflexive.
\end{corollary}
\begin{proof}
By the spectral theorem for normal tuples \cite{AM}, $\mathscr M_z$ satisfies von Neumann's inequality. Now apply Theorem \ref{reflexive}.
\end{proof}

A celebrated result of Brown and Chevreau \cite{BC} states that any contraction with isometric $H^{\infty}$-functional calculus is reflexive (see \cite[Chapter III, Theorem 11.3]{SF} for exact statement).
The following provides a sufficient condition for reflexivity of polynomially bounded operators (cf. \cite[Proposition 4.4]{CEP}, \cite[Theorem A]{AMu}). Recall that $T \in B(\mathcal H)$ is {\it polynomially bounded} if there exists a  constant $M >0$ such that
$$ \|p(T)\|_{B(\mathscr H)} \Le M \|p\|_{\infty, \mathbb D}~\mbox{for every}~p \in \mathbb C[z].$$
\begin{corollary}
Any left-invertible, analytic polynomially bounded $T$ in $B(\mathcal H)$ is reflexive provided the spectral radius of the Cauchy dual operator $T'$ is at most $1$.
\end{corollary}
\begin{proof}
This is immediate from Theorem \ref{reflexive} and Shimorin's analytic model for left-invertible analytic operators \cite{Shimorin}, where the assumption that $r(T') \Le 1$ ensures that the $E$-valued functions in the model space $\mathscr H$ of $T$ are holomorphic in the open unit disc $\mathbb D$.
\end{proof}
The last corollary is applicable to any analytic operator $T$ in $B(\mathcal H)$ satisfying the following inequality:
\beqn
\|Tx+y\|^2 \Le 2(\|x\|^2+\|Ty\|^2), \quad x, y \in \mathcal H.
\eeqn
In fact, an examination of the proof of \cite[Theorem 3.6]{Shimorin} shows that the Cauchy dual $T'$ of $T$ exists and
satisfies $I-2T'^{*}T' + T'^{*2}T'^2 \Le 0$. It can be concluded from \cite[Lemma 1]{Ri} that $T$ is a contraction and the spectral radius of $T'$ is at most $1.$ 
\section{Applications to weighted shifts on rooted directed trees}

The reader is referred to \cite{Jablonski} for all the relevant definitions pertaining to the rooted directed trees and associated weighted shifts. 
Let $\mathscr T = (V, \mathcal E)$ be a leafless, rooted directed tree and let $$V_\prec := \{v \in V : \mbox{card}(\child{v}) \Ge 2\}$$ denote the set of branching vertices. The {\it branching index} $k_{\mathscr T}$ of $\mathscr T$ is defined as
\[k_\mathscr{T}:=\begin{cases}
 1+\sup\{\dep_w:w\in V_{\prec}\}& \text{if $V_{\prec}$ is non-empty},\\
 0& \text{otherwise},
\end{cases}
\] 
where $\dep_w$ is the unique non-negative integer such that $w \in \mathsf{Chi}^{\langle \dep_w\rangle}(\mathsf{root})$ (see \cite[Corollary 2.1.5]{Jablonski}). We refer to $\dep_w$ as the {\it depth} of $w$ in $\mathscr T.$ Let $S_\lambda$ be a weighted shift on a rooted directed tree $\mathscr T$. Then $E := \ker(S^*_\lambda)$ is finite dimensional if and only if $\mathscr T$ is locally finite with finite branching index (see \cite[Proposition 2.1]{CT}). 

Let $\mathscr T=(V, \mathcal E)$ be a leafless, locally finite rooted directed tree. 
For an integer $a \Ge 2,$ the {\it Bergman shift} $\mathscr B_a$ is the weighted shift on $\mathscr T$ with weights given by
\beq \label{wts} 
\lambda_{u, a} = 
 \frac{1}{\sqrt{\mbox{card}(\child{v})}}\,\sqrt{\frac{\dep_v+1}{\dep_v + a}}, \quad u \in \child{v},~v \in V,
\eeq
where $\dep_v$ is the depth of $v$ in $\mathscr T.$
Needless to say, the shift $\mathscr B_2$ is the {\it Bergman shift} if $\mathscr T$ is the rooted directed tree without any branching vertex. By \cite[Proposition 5.1.8]{CPT}, $\mathscr B_a$ is unitarily equivalent to the operator $\mathscr M_{z, a}$ of multiplication by the coordinate function $z$ on $\mathscr H_{a}$,
where $\mathscr H_{a}$ is the reproducing kernel Hilbert space associated with the reproducing kernel $\kappa_{\mathscr H_{a}} : \mathbb D \times \mathbb D \rar B(E)$ given by
\beq
\label{Berg-k}
\kappa_{\mathscr H_a}(z, w) 
&=& \sum_{n=0}^{\infty}{n+a-1 \choose n}~ 
z^n \overline{w}^n\,P_{[e_{\mathsf{root}}]}   \\ &+& \notag
\sum_{v \in V_{\prec}} \sum_{n=0}^{\infty}  
\frac{(\dep_v + n + a)! (\dep_v +1)!}{(\dep_v +a)! (\dep_v + n+1)!}\,
z^n\overline{w}^n\,
P_{l^2(\child{v}) \ominus [\Gamma_v]}, \quad z, w \in \mathbb D.
\eeq
Here $E=\ker(\mathscr B^*_a)$ and $\Gamma_v : \mathsf{Chi}(v) \rar \mathbb C$ 
is given by 
$\Gamma_v = \sum_{u \in \mathsf{Chi}(v)} \lambda_u e_u$. Clearly, $\kappa_{\mathscr H_a}(\lambda, 0) = I_E$ for any $\lambda \in \mathbb D$. Further, it can be easily seen that
\beq \label{moment}
\|z^ng\|^2 =\frac{(\dep_v +a-1)! (\dep_v + n)!}{(\dep_v + n + a-1)! \dep_v !}, \quad g \in l^2(\childn{\dep_v}{\mathsf{root}}).
\eeq

\begin{proposition} \label{comm-B}
Let $\mathscr T=(V, \mathcal E)$ be a leafless, locally finite rooted directed tree and let $\mathscr B_a$ be the Bergman shift on $\mathscr T$. If $\mathscr T$ has finite branching index, then the commutant $\{\mathscr B_a\}'$ of $\mathscr B_a$ is isometrically isomorphic to $(H^{\infty}_{_{B(E)}}(\mathbb D), \|\cdot\|_{B(\mathscr H)})$,
where the finite dimensional Hilbert space $E$ equals the kernel of $\mathscr B^*_a$.
\end{proposition}
\begin{proof}
Suppose that $\mathscr T$ has finite branching index.
Note that
\beqn
 \sup_{v \in V}\sum_{u \in \child{v}} \lambda^2_{u, a} \overset{\eqref{wts}}= \sup_{v \in V}\frac{\dep_v+1}{\dep_v + a} =1,
 \eeqn
 and hence by \cite[Proposition 3.1.8]{Jablonski}, $\mathscr B_{a}$ is a contraction. 
 Hence, in view of Theorem \ref{commutant} and the discussion prior to the statement of Proposition \ref{comm-B}, it suffices to check that the reproducing kernel $\kappa_{\mathscr H_a}$, as given by \eqref{Berg-k}, satisfies \eqref{assumption} and \eqref{g-rate}. 
 To see \eqref{assumption}, let $f(z) = \sum_{n=0}^{\infty} a_n z^n \in \mathscr H$, where $\{a_n\}_{n \in \mathbb N} \subseteq E.$ Let $\{g_1, \ldots, g_{\dim E}\}$ be an orthonormal basis of $E$ such that 
 for any $i=1, \ldots, \dim E,$ \beq \label{choice} g_i \in l^2(\childn{k_i}{\mathsf{root}})~ \mbox{for some} ~k_i \in \mathbb N. \eeq
 It is easy to see using \eqref{moment} that $\inp{z^ng_j}{z^ng_i}=0$ if $i \neq j$.
Also, since $\{z^nE\}_{n \in \mathbb N}$ are mutually orthogonal (\cite[Lemma 5.2.7]{CPT}), we have
\beq
\label{member}
\|f\|^2 &=& \sum_{n=0}^{\infty}\|z^n a_n\|^2=\sum_{n=0}^{\infty}\|\sum_{j=1}^{\dim E} \inp{a_n}{g_j}_E z^ng_j\|^2 \notag \\ &=& \sum_{j=1}^{\dim E} \sum_{n=0}^{\infty}|\inp{a_n}{g_j}_{_E}|^2\|z^ng_j\|^2.
\eeq 
Note further that
 \beqn
\inp{f(w)}{g_j}_{_E} g_k =\sum_{n=0}^{\infty} \inp{a_n}{g_j}_{_E}w^ng_k, \quad w \in \mathbb D,
 \eeqn
and hence by the mutual orthogonality of $\{z^nE\}_{n \in \mathbb N}$,
we obtain
\beqn
\|\inp{f(\cdot)}{g_j}_{_E} g_k\|^2 =\sum_{n=0}^{\infty} |\inp{a_n}{g_j}_{_E}|^2\|z^ng_k\|^2 = \sum_{n=0}^{\infty} |\inp{a_n}{g_j}_{_E}|^2\|z^ng_j\|^2 \frac{\|z^ng_k\|^2}{\|z^ng_j\|^2}.
\eeqn
To complete the verification of \eqref{assumption}, in view of \eqref{member}, it suffices to check that the sequence $\{{\|z^ng_k\|^2}/{\|z^ng_j\|^2}\}_{n \in \mathbb N}$ is bounded. Indeed, this sequence is convergent in view of \eqref{moment} and \eqref{choice}.

 To see \eqref{g-rate}, fix $w \in \mathbb D,$ and note that $\kappa_{\mathscr H_a}(w, w)$ is a positive diagonal operator with respect to the orthonormal bases of $[e_{\mathsf{root}}]$ and $l^2(\child{v}) \ominus [\Gamma_v],$ $v \in V_{\prec}.$ Moreover, the  diagonal entries of $\kappa_{\mathscr H_a}(w, w)$ are given by
 \beqn
 \sum_{n=0}^{\infty}{n+a-1 \choose n}~ 
|w|^{2n}, \quad \sum_{n=0}^{\infty}  
\frac{(\dep_v + n + a)! (\dep_v +1)!}{(\dep_v +a)! (\dep_v + n+1)!}\,|w|^{2n}, \quad v \in V_{\prec}.
 \eeqn
Consider the bi-sequence $\{a_{m, n}\}_{m, n \in \mathbb N}$ given by
 \beq \label{biseq}
 a_{m, n} = \frac{(m + n + a)! (m +1)!}{(m +a)! (m + n+1)!}, \quad m, n \in \mathbb N.
 \eeq
Since $a \Ge 2,$  
$\{a_{m, n}\}_{m \in \mathbb N}$ is decreasing for every $n \in \mathbb N.$ Further, since \beqn    
{n+a-1 \choose n} \Ge a_{0, n},  \quad n \in \mathbb N,\eeqn 
it follows that the minimum $\mu_{\min}(w)$ and maximum $\mu_{\max}(w)$ of eigenvalues of $\kappa_{\mathscr H_a}(w, w)$ are given respectively by
\beqn
\mu_{\min}(w) &=& \sum_{m =0}^{\infty} a_{m_0, n} |w|^{2n}, \quad m_0 := \max\{\dep_v : v \in V_{\prec}\}, \\
\mu_{\max}(w) &=& \sum_{n=0}^{\infty}{n+a-1 \choose n}|w|^{2n},
\eeqn
where $m_0$ is finite since $\mathscr T$ has finite branching index.
Thus \eqref{g-rate} is equivalent to   
\beq
\label{max-min-finite}
\sup_{w \in \mathbb D} \frac{\mu_{\max}(w)}{\mu_{\min}(w)} < \infty.
\eeq
To see \eqref{max-min-finite},
note that $$\lim_{n \rar \infty} \frac{{n+a-1 \choose n}}{a_{m_0, n}}=
\frac{(m_0+a)!}{(m_0+1)!(a-1)!}.$$
It follows now from \eqref{biseq} that there exists a positive integer $n_0$ such that $${n+a-1 \choose n}  \Le (m_0+a)!\,{a_{m_0, n}}, \quad n \Ge n_0.$$ Consequently,
\beqn
\frac{\mu_{\max}(w)}{\mu_{\min}(w)} &=&  \frac{ \sum_{n =0}^{\infty}  {n+a-1 \choose n} |w|^{2n}}{\sum_{n=0}^{\infty} a_{m_0, n} |w|^{2n}} \\ & \Le & \frac{\sum_{n=0}^{n_0} {n+a-1 \choose n} |w|^{2n}}{\sum_{n=0}^{\infty}a_{m_0, n} |w|^{2n}} + \frac{\sum_{n =n_0+1}^{\infty} (m_0+a)!~ a_{m_0, n} |w|^{2n}}{\sum_{n=0}^{\infty}a_{m_0, n}|w|^{2n}} \\ & \Le &  \sum_{n =0}^{n_0} {n+a-1 \choose n} |w|^{2n} + (m_0+a)!, 
\eeqn
and hence we obtain the conclusion in \eqref{max-min-finite}.
\end{proof}
\begin{remark}
An examination of the proof shows that there exists a real polynomial $p$ such that
\beq
\label{polynomial-g}
\displaystyle \|\kappa_{\mathscr H_a}(w, w)\|_{_{B(E)}} \|\kappa_{\mathscr H_a}(w, w)^{-1}\|_{_{B(E)}} \Le p(|w|^2), \quad w \in \mathbb D.
\eeq
\end{remark}

The following corollary is immediate from Corollary \ref{comm-abelian}, Proposition \ref{comm-B} and \cite[Proposition 3.5.1]{Jablonski}.
\begin{corollary} Let $\mathscr T=(V, \mathcal E)$ be a leafless, locally finite rooted directed tree and let $\mathscr B_a$ be the Bergman shift on $\mathscr T$. If $\mathscr T$ has finite branching index then the commutant $\{\mathscr B_a\}'$ of $\mathscr B_a$ is abelian if and only if $\mathscr T$ is graph isomorphic to the rooted directed tree without any branching vertex.
\end{corollary}

It would be of independent interest to characterize $B(E)$-valued reproducing kernels $\kappa$ which satisfy \eqref{polynomial-g} for a polynomial $p$. In the context of Bergman shifts on $\mathscr T$, this problem seems to be closely related to the notion of finite branching index of $\mathscr T$.
One may further ask for a multivariable counter-part of Proposition \ref{comm-B}. We believe that similar arguments can be used to obtain a counter-part of Proposition \ref{comm-B} for multivariable analogs $S_{\lambdab_{\mf c_a}}$ of Bergman shifts $\mathscr B_a$ (refer to \cite[Chapter 5]{CPT} for the definition of $S_{\lambdab_{\mf c_a}}$). 
Further, it may be concluded from Corollary \ref{coro-list2} and \cite[Theorem 5.2.6 and Example 5.3.5]{CPT} that the $d$-tuple $S_{\lambdab_{\mf c_a}}$ is reflexive for any integer $a \Ge d.$ 
In order to avoid book-keeping, we skip these verifications. Here we discuss one family of weighted multishift to which Theorem \ref{reflexive} is applicable.  

The following can be seen as a $2$-variable counterpart of \cite[Theorem 10]{BDPP} (the reader is referred to \cite{CPT} for the definitions of directed Cartesian product of directed trees and associated multishifts).
\begin{proposition} \label{multi}
Let $\mathscr T = (V,\mathcal E)$ be the directed Cartesian
product of locally finite, rooted directed trees $\mathscr T_1,  \mathscr T_2$ and let $S_{\lambdab}=(S_1,  S_2)$ be a multishift on $\mathscr T$ consisting of left-invertible operators $S_1$ and $S_2$. Let $E$ be the joint kernel of $S^*_{\lambdab}$.
Assume that $S_{\lambdab}$ satisfies 
\beq \label{k-c-two} E \subseteq \ker(S_1^* S'^{\alpha_2}_2) \cap \ker(S_2^* S'^{\alpha_1}_1)~ \mbox{for all}~ (\alpha_1, \alpha_2) \in \mathbb{N}^2, \eeq
and that the $2$-tuple $(S'_1, S'_2)$ consists of commuting Cauchy dual operators of spectral radii at most $1$.
If $S_{\lambdab}$ is contractive, then it is reflexive.
\end{proposition}
\begin{proof}
By \cite[Theorem 4.2.4]{CPT}, there exist a reproducing kernel Hilbert space $\mathscr H$ of $E$-valued holomorphic functions defined on the unit bidisc centered at the origin and a unitary $U : l^2(V) \rar \mathscr H$ such that $U S_j = \mathscr M_{z_j} U$ for $j=1, 2$. 
The desired conclusion now follows from Theorem \ref{reflexive}, List \ref{list}, and Lemma \ref{dilation-von}.
\end{proof}
\begin{remark} \label{rmk-multi}
In case any one of $\mathscr T_1$ and $\mathscr T_2$ is without any branching vertex, \eqref{k-c-two} always holds (see \cite[Corollary 4.2.9]{CPT}). 
\end{remark}
%

\subsection{A two-parameter family of tridiagonal $B(E)$-valued kernels}

We conclude the paper by exhibiting a two parameter family of a {\it tridiagonal} $B(E)$-valued kernels, which satisfy the assumptions of Theorems \ref{commutant} and \ref{reflexive}. 
Consider the rooted directed tree $\mathscr T=(V, \mathcal E)$ as discussed in \cite[Section 6.2]{Jablonski} (see also \cite[Example 3]{CT}). Recall that the set $V$ of vertices of $\mathscr T$ is given by
$$V:=\{(0,0)\}\cup\{(1,i), (2,i):i\Ge1\}$$ with
$\mathsf{root}=(0,0),$ and the edges are governed by $\mathsf{Chi}(0,0)=\{(1,1),(2,1)\}$ and 
\[\mathsf{Chi}(1,i)=\{(1,i+1)\},\
\mathsf{Chi}(2,i)=\{(2,i+1)\}, \quad i\Ge1. \] For positive numbers $\mathfrak s$ and $\mathfrak t$ with $\mf t \neq 1$, consider the weight system $\lambda_{\mathfrak s, \mathfrak t}$  given by  
\begin{align} \label{wts-1}
 \left.
  \begin{minipage}{36ex}
\beqn 
\lambda_{(1, 1)} &=  \mathfrak s =& \lambda_{(2, 1)}, \\
\lambda_{(1, 2)} &= 1  =& \lambda_{(2, 3)}, \\
  \lambda_{(2, 2)} &= \mathfrak t  = & \lambda_{(1, 3)}, \\
  \lambda_{(j, i)} &  = 1,   &  j  =1, 2, ~i  \Ge  4.
  \eeqn
 \end{minipage}
   \right\} 
\end{align}
Let $S_{\lambda_{\mathfrak s, \mathfrak t}}$ be the weighted shift with weight system $\lambda_{\mathfrak s, \mathfrak t}$ and let $E:= 
\ker(S^*_{\lambda_{\mathfrak s, \mathfrak t}})$. Then, as noted in \cite[Proposition 4.1]{CT}, $S_{\lambda_{\mathfrak s, \mathfrak t}}$ 
is unitarily equivalent to the multiplication operator $\mathscr M_z$ on the reproducing kernel Hilbert space $\mathscr H$ of $E$-valued holomorphic functions on the unit disc $\mathbb D$. The reproducing kernel $\kappa_\mathscr H : \mathbb D 
\times \mathbb D \rar B(E)$ of $\mathscr H$ is given by 
\beq \label{kernel}
\kappa_{\mathscr H}(z,w) &=& I_E + \alpha_0(x \otimes y\, z^2\overline{w} + y \otimes x\, z\overline{w}^2) \notag \\ &+&     \sum_{k=1}^{\infty} \Big({\alpha_k}\, x \otimes x + {\alpha_{k+1}}\, y \otimes y \Big) z^k \overline{w}^k, \quad
z, w \in \mathbb D,
\eeq
where $x = e_{(0, 0)}$, $y=\mf s(e_{(1, 1)}-e_{(2, 1)})$ are orthogonal basis vectors for $E$,   and 
\beqn \alpha_{k}:= \begin{cases}
\frac{1}{4\mf s^2}\left(1 - \mf t^{-2}\right)&~\mbox{if~}k=0,\\
\frac{1}{2\mf s^2}&~\mbox{if~}k=1,\\
\frac{1}{4\mf s^2}( 1 +\mf t^{-2})&~\mbox{if~}k=2,\\
\frac{1}{2\mf s^2 \mf t^{2}}&~\mbox{if~} k \Ge 3.\end{cases}
\eeqn 
Clearly, $\kappa_{\mathscr H}$ satisfies the normalization condition \eqref{nc}.
One may argue as in \cite[Example 4]{CT} to deduce that
\beq
\label{onb}
\mathcal B:=\{x\} \cup \left \{a_k z^{k}p(z)\right \}_{k \in \mathbb N} \cup  \left 
\{b_kz^{k}q(z)\right \}_{k \in \mathbb N}
\eeq
forms an orthonormal basis of $\mathscr H$,  
where 
\beq \label{ak-bk} \mbox{$a_0=a_1=1$, $a_k=\frac{1}{\mf t}$, $k \Ge 2$, \quad $b_0=1$,  $b_k=\frac{1}{\mf t}$, $k \Ge 1$},
\eeq
\beq \label{p-q}
\mbox{$p(z) = \frac{1}{2\mf s}( x z + y),$ \quad $q(z) = \frac{1}{2\mf s}(xz -y)$.}
\eeq


\begin{lemma} \label{above}
The $B(E)$-valued kernel $\kappa_{\mathscr H}$, as given by \eqref{kernel},
satisfies the conditions \eqref{assumption} and \eqref{g-rate} of Theorem \ref{commutant}.
\end{lemma}
\begin{proof}
 Let $f \in \mathscr H$. Since $\mathcal B$, as given by \eqref{onb}, forms an orthonormal basis for $\mathscr H,$  
there exist $c, c_k, d_k \in \mathbb C$, $k \Ge 0$, such that
\beqn
f(z) = c x +  \sum_{k=0}^\infty c_k a_k {z^k p(z)} + \sum_{k=0}^\infty d_k b_k {z^k q(z)}, \quad z \in \mathbb D.
\eeqn
It follows that
$\|f\|^2_\mathscr H = |c|^2 + \sum_{k=0}^\infty |c_k|^2 + \sum_{k=0}^\infty |d_k|^2.$
That is, $\{c_k\}_{k \Ge 0}$, $\{d_k\}_{k \Ge 0}$ are in $l^2(\mathbb N)$. For $g, h \in E$, define
$f_{g,h}(w) = \inp{f(w)}{g}_E h, \quad w \in \mathbb D.$
Since $\{x,y\}$ is an orthogonal basis of $E$, in order to show that $f_{g,h} \in \mathscr H$, it suffices to check that $f_{x,y}, ~f_{y,x}, ~f_{x,x}, ~f_{y,y} \in \mathscr H$.  
It is easy to see using \eqref{p-q}  that
\beqn
f_{x,y}(w) 
= c y + \frac{1}{2\mf s} \sum_{k=0}^\infty (a_k {c_k} +b_k d_k) w^{k+1}y, \quad w \in \mathbb D.
\eeqn
Since $\{z^ny\}_{n \in \mathbb N}$ is orthogonal and for $n \in \mathbb N,$ \beqn \|z^ny\|^2_{\mathscr H} &=& \mf s^2 \|S^n_{\lambda_{\mathfrak s, \mathfrak t}}(e_{(1, 1)}-e_{(2, 1)})\|^2_{l^2(V)} \\ &=& \mf s^2 \big (\|S^n_{\lambda_{\mathfrak s, \mathfrak t}}e_{(1, 1)}\|^2_{l^2(V)}+ \|S^n_{\lambda_{\mathfrak s, \mathfrak t}}e_{(2, 1)})\|^2_{l^2(V)} \big)\\
& \Le & \max \{2\mf s^2, 2\mf s^2 \mf t^2, \mf s^2(1+\mf t^2)\},\eeqn it follows from $\{c_k\}_{k \in \mathbb N}$, $\{d_k\}_{k \in \mathbb N} \in l^2(\mathbb N)$ and \eqref{ak-bk} that 
$f_{x, y} \in \mathscr H.$
Along the similar lines, one can check that $f_{y,x}, ~f_{x,x}, ~f_{y,y} \in \mathscr H$.  
This yields \eqref{assumption}.

To see \eqref{g-rate}, note that $\|y\|=\mf s\sqrt{2}.$ Thus, for $w \in \mathbb D$, by \eqref{kernel}, we have
\beq \label{k1}
\kappa_{\mathscr H}(w,w) x &=& k_1(w,w) x + \alpha_0\, {\mf s \sqrt{2}} ~w\overline{w}^2 \frac{y}{\|y\|} \notag \\
 k_1(w,w) &=& 1+\sum_{k=1}^{\infty} \alpha_k |w|^{2k}.\eeq 
 Similarly, for $w \in \mathbb D,$ we have
\beq \label{k2}
\kappa_{\mathscr H}(w,w) \frac{y}{\|y\|} &=& \alpha_0\, {\mf s \sqrt{2}} ~ w^2\overline{w} ~x + k_2(w,w) \frac{y}{\|y\|},\notag \\
 k_2(w,w) &=& 1+ \sum_{k=1}^{\infty} \alpha_{k+1} 2 \mf s^2 |w|^{2k}.\eeq 
The matrix representation, say $A(w)$, of the positive operator $\kappa_{\mathscr 
H}(w,w)$ with respect to the basis $\{x, \frac{y}{\|y\|}\}$ is given by
\beqn
A(w) = \left[\begin{array}{cc}
k_1(w,w) & a w^2\overline{w}\\
a w\overline{w}^2 & k_2(w,w)
\end{array}\right],
\eeqn
where $a:=\alpha_0\, {\mf s \sqrt{2}}.$
The eigenvalues $x_+(w)$ and $x_-(w)$ of $A(w)$ are given by
\beqn
x_{\pm}(w) = \frac{1}{2}\left(k_1(w,w) + k_2(w,w) \pm \sqrt{\big(k_1(w,w)-k_2(w,w)\big)^2 + 4a^2 |w|^6}\right).
\eeqn
Clearly, $x_+(w) \Ge x_-(w)$ for all $w \in \mathbb D$. It follows that for any $w \in \mathbb D$,
\beqn
\frac{x_+(w)}{x_-(w)} &=& \frac{k_1(w,w) + k_2(w,w) + \sqrt{\big(k_1(w,w)-k_2(w,w)\big)^2 + 4a^2 |w|^6}}{k_1(w,w) + 
k_2(w,w) - \sqrt{\big(k_1(w,w)-k_2(w,w)\big)^2 + 4a^2 |w|^6}}\\
&\Le& \frac{k_1(w,w) + k_2(w,w) + |k_1(w,w)-k_2(w,w)| + 2a |w|^3}{k_1(w,w) + 
k_2(w,w) - |k_1(w,w) - k_2(w,w)| - 2a|w|^3}\\
&\Le& \max \left\{ \frac{ k_1(w,w) + a}{ k_2(w,w) - a}, ~\frac{ k_2(w,w) + a}{ k_1(w,w) - a}\right\},
\eeqn
which, in view of \eqref{k1} and \eqref{k2}, is easily seen to be of polynomial order as a function of $|w|^2.$
This completes the verification of \eqref{g-rate}.
\end{proof}
Assume that $\mf s \in (0, 1/\sqrt{2}]$ and $\mf t \in (0, 1).$ By \cite[Proposition 3.1.8]{Jablonski}, $S_{\lambda_{\mathfrak s, \mathfrak t}}$ is a contraction.
Combining Lemma \ref{above} with Theorem \ref{commutant}, we conclude that the commutant of $S_{\lambda_{\mathfrak s, \mathfrak t}}$ is isometrically isomorphic to $(H^{\infty}_{_{B(E)}}(\mathbb D), \|\cdot\|_{B(\mathscr H)})$, where $E=\ker(S^*_{\lambda_{\mathfrak s, \mathfrak t}})$ is the $2$-dimensional space spanned by $x$ and $y$. Further, by Theorem \ref{reflexive}, $S_{\lambda_{\mathfrak s, \mathfrak t}}$ is reflexive.
It is worth mentioning that $S_{\lambda_{\mathfrak s, \mathfrak t}}$ is never {\it hyponormal}, that is, $S^*_{\lambda_{\mathfrak s, \mathfrak t}}S_{\lambda_{\mathfrak s, \mathfrak t}}-S_{\lambda_{\mathfrak s, \mathfrak t}}S^*_{\lambda_{\mathfrak s, \mathfrak t}} \ngeqslant 0$  
(see \cite[Theorem 5.1.2]{Jablonski}).


\medskip \textit{Acknowledgment}. \
A part of this paper was written while the second author visited the Department of Mathematics and Statistics, IIT Kanpur. He expresses his gratitude to the faculty and the administration of this unit for their warm hospitality. The authors appreciate the suggestions of Md. Ramiz Reza and Deepak Kumar Pradhan pertaining to the definition of the functional Hilbert space improving the earlier presentation.
The authors would like to thank Jan Stochel for his continual support and  encouragement. Finally, the authors are grateful to the anonymous referee for several important remarks (see Remark \ref{Hinfty} and Remark \ref{rmk-reflexive}(3)). In particular, it has been pointed out by the referee that Theorem \ref{commutant} can be recovered from its scalar-valued counter-part (the case of $\dim\, E=1$).

\end{document}